\documentclass[a4paper, reqno]{amsart}
\usepackage[margin=3.81cm]{geometry}
\usepackage[utf8]{inputenc}
\usepackage[english]{babel}

\usepackage{csquotes}
\usepackage{graphicx}
\usepackage[utf8]{inputenc}
\usepackage[T1]{fontenc}
\usepackage[dvipsnames]{xcolor}% needs to be loaded before todonotes, since todonotes loads xcolor without options.
\usepackage{subfiles}% allows us to compile individual sections

\usepackage[textsize = footnotesize]{todonotes}

\usepackage[mathscr]{eucal}
\usepackage[shortlabels]{enumitem}
\usepackage[super]{nth}
\usepackage{leftindex}

\usepackage[backend=biber,
    style=numeric,
    sorting=nyt,% sort by name, year, title
    giveninits=true,% render all first names as initials.
    isbn=false,
    maxalphanames=4,
    date=year,
    maxbibnames=99,
    backref=false]{biblatex}
\DeclareRobustCommand{\SkipTocEntry}[5]{}

\usepackage{amsmath}
\usepackage{amssymb}

\usepackage{amsthm}
\usepackage{thmtools}
\usepackage{graphicx}

\usepackage{dsfont} %blackboard bold 1
\usepackage{mathtools}
\usepackage{xparse}
\usepackage{tikz-cd}
\tikzset{%adjunction symbol
    symbol/.style={%
        draw=none,
        every to/.append style={%
            edge node={node [sloped, allow upside down, auto=false]{$#1$}}}
    }
}
\usepackage{spectralsequences}

\definecolor{leafgreen}{RGB}{48, 138, 3}

\usepackage[pdfborder={0 0 0}]{hyperref}
\hypersetup{
   colorlinks=true,
   allcolors=.,
   bookmarksdepth = 3
}
\usepackage[capitalize, nameinlink]{cleveref}
\crefformat{equation}{#2(#1)#3}
\Crefformat{equation}{#2(#1)#3}
\crefname{section}{Section}{Sections}
\crefname{subsection}{Section}{Sections}
\crefname{subsubsection}{Section}{Sections}
\crefname{subappendix}{Section}{Sections}
\crefname{subsubappendix}{Section}{Sections}
\usepackage{microtype}

\theoremstyle{plain}
\newtheorem{theorem}{Theorem}[section]
\newtheorem{introtheorem}{Theorem}

\newtheorem{lemma}[theorem]{Lemma}
\newtheorem{proposition}[theorem]{Proposition}

\newtheorem*{theorem*}{Theorem}

\theoremstyle{definition}
\newtheorem{definition}[theorem]{Definition}
\newtheorem{remark}[theorem]{Remark}
\newtheorem{warning}[theorem]{Warning}

\newtheorem{construction}[theorem]{Construction}

\theoremstyle{remark}

\DeclareMathOperator{\id}{id}
\DeclareMathOperator*{\colim}{colim}

\DeclareMathOperator{\triv}{triv}

\DeclareMathOperator{\BMod}{BMod}

\DeclareMathOperator{\BcoMod}{BcoMod}

\DeclareMathOperator{\Alg}{Alg}
\DeclareMathOperator{\coAlg}{coAlg}

\DeclareMathOperator{\Sym}{Sym}

\DeclareMathOperator{\coOp}{coOp}

\newcommand{\hide}[1]{}

\newlist{numberenum}{enumerate}{1}
\setlist[numberenum]{\upshape(\arabic*)}

\newcommand{\Fun}{\mathrm{Fun}}

\newcommand{\Lax}{\mathrm{Lax}}
\newcommand{\Laxunit}{\mathrm{Lax}^\mathrm{un}}
\newcommand{\Oplax}{\mathrm{Oplax}}
\newcommand{\Oplaxunit}{\mathrm{Oplax}^{\mathrm{un}}}
\newcommand{\FunL}{\mathrm{Fun}^{\mathrm{L}}}

\newcommand{\FunLM}{\Fun_{\lten}}
\newcommand{\FunRM}{\Fun_{\rten}}
\newcommand{\FunBM}{\Fun_{\bten}}
\newcommand{\iFunLM}{\underline\Fun_{\lten}}
\newcommand{\iFunRM}{\underline\Fun_{\rten}}
\newcommand{\FunLL}[1]{\FunLM}
\newcommand{\FunRR}[1]{\FunRM}
\newcommand{\FunBB}[2]{\FunBM}
\newcommand{\iFunL}[1]{\iFunLM}
\newcommand{\iFunR}[1]{\iFunRM}

\newcommand{\Fin}{\mathrm{Fin}}

\newcommand{\sA}{\mathscr{A}}
\newcommand{\sB}{\mathscr{B}}
\newcommand{\sC}{\mathscr{C}}
\newcommand{\sD}{\mathscr{D}}

\newcommand{\sP}{\mathscr{P}}
\newcommand{\sQ}{\mathscr{Q}}

\newcommand{\sO}{\mathscr{O}}

\newcommand{\sR}{\mathscr{R}}
\newcommand{\sS}{\mathscr{S}}

\newcommand{\sX}{\mathscr{X}}
\newcommand{\sY}{\mathscr{Y}}

\newcommand{\Sp}{\mathrm{Sp}}

\newcommand{\diffst}{\mathscr{D}\mathrm{iff}_{\mathrm{St}}}

\newcommand{\presl}{\mathscr{P}\mathrm{r}^{\mathrm{L}}}

\newcommand{\pressymstc}{\mathscr{P}\mathrm{r}^{\mathrm{Sym}}_{\mathrm{St}, \mathrm{nu}}}

\newcommand{\Cat}{\mathscr{C}\mathrm{at}}

\newcommand{\MMor}{\mathscr{M}\mathrm{or}}
\newcommand{\bMMor}{\mathbb{M}\mathrm{or}}
\newcommand{\coMMor}{\mathrm{co}\mathscr{M}\mathrm{or}}

\newcommand{\op}{\mathrm{op}}
\newcommand{\co}{\mathrm{co}}

\newcommand{\SSeq}{\mathrm{SSeq}}
\newcommand{\SSeqnu}{\mathrm{SSeq}^{\mathrm{nu}}}
\newcommand{\SSeqsp}{\mathrm{SSeq}_{+}}

\newcommand{\Exc}{\mathrm{Exc}}

\newcommand{\lin}{\mathrm{P}_1}
\newcommand{\KD}{\mathrm{KD}}

\newcommand{\reAlg}{\Lambda_{\mathrm{Alg}}}

\newcommand{\diffalg}{\mathrm{Diff}_{\mathrm{Alg}}}

\newcommand{\Sph}{\mathbb{S}}
\newcommand{\unit}{\mathbf{1}}

\makeatletter
\newcommand{\oset}[3][0ex]{%
	\mathrel{\mathop{#3}\limits^{
			\vbox to#1{\kern-2\ex@
				\hbox{$\scriptstyle#2$}\vss}}}}
\makeatother

\tikzset{
	rot90/.style={anchor=south, rotate=90, inner sep=.5mm}
} % Rotating the \sim symbol in a tikz label bij 90 degrees.

\NewDocumentCommand\derprojlim{e{_}}{\mathchoice
{\varprojlim  \IfValueT{#1}{_{\mathclap{#1}}}{}^{\!1\!}\mathop{}}
{\varprojlim^1  \IfValueT{#1}{_{#1}}}
{\varprojlim^1  \IfValueT{#1}{_{#1}}}
{\varprojlim^1  \IfValueT{#1}{_{#1}}}}

\newcommand{\Op}{\mathrm{Op}}

\newcommand{\lten}{\mathrm{LM}}
\newcommand{\rten}{\mathrm{RM}}
\newcommand{\bten}{\mathrm{BM}}

\makeatletter
 \def\subsection{\@startsection{subsection}{1}%
 \z@{.7\linespacing\@plus\linespacing}{.5\linespacing}%
 {\normalfont\bfseries\centering}}% NEW
\makeatother

\renewcommand{\epsilon}{\varepsilon}

\usepackage{calc}

\newcommand{\myarrow}[1]{%
  \xrightarrow{%
    \mathmakebox[\widthof{$\scriptstyle\Sigma^\infty_{\sP} \circ - \circ \Omega^\infty_{\sO}$}][c]{#1}%
  }%
}

\begin{document}

\title{Koszul duality and Morita categories}

\author{Max Blans}
\address{University of Oxford}
\email{max.blans@maths.ox.ac.uk}

\begin{abstract}
We prove an $(\infty,2)$-categorical version of Koszul duality for operads and cooperads in spectra by showing that there is an equivalence between a Morita category of operads and bimodules and a dual Morita category of cooperads and bicomodules.
This result subsumes various forms of operadic Koszul duality 
present in the literature.
% Our result is phrased in terms of the Morita category $\MMor$.
% This is an $(\infty, 2)$-category whose objects are spectral operads and whose $1$-morphisms are bimodules.
% There is also a dual Morita category $\MMor^{\vee}$, whose objects are spectral cooperads and whose $1$-morphisms are bicomodules.
% We show that Koszul duality gives rise to an equivalence $\MMor \simeq \MMor^{\vee}$ of $(\infty,2)$-categories.
\end{abstract}

\maketitle

\tableofcontents

\section{Introduction}

Koszul duality is a correspondence between operads and cooperads first introduced by Ginzburg--Kapranov \cite{GinzburgKapranov} and Getzler--Jones \cite{getzler1994operadshomotopyalgebraiterated} in the setting of chain complexes, and later extended to spectra by Ching \cite{ChingThesis} and Salvatore \cite{Salvatore-Thesis}.
The theory was worked out $\infty$-categorically and further generalized by Lurie in \cite[\S 5.2]{HA}.

The duality takes the form of an adjoint equivalence
\[
\begin{tikzcd}
B \colon \Op(\Sp) \ar[r, shift left] & \coOp(\Sp) \colon C \ar[l, shift left]
\end{tikzcd}
\]
between the $\infty$-categories of strongly positive\footnote{A (co)operad $\sO$ is strongly positive if $\sO(0)\simeq \ast$ and $\sO(1) \simeq \Sph$.} operads and cooperads in spectra; the fact that this is an equivalence of categories is due to Ching \cite{ChingBarCobar}.
For each pair $\sO, \sP \in \Op(\Sp)$, we also obtain a Koszul duality adjunction
\[
\begin{tikzcd}
B \colon \BMod_{(\sO, \sP)} \ar[r, shift left] & \BcoMod_{(B\sO, B\sP)} \colon C \ar[l, shift left] 
\end{tikzcd}
\]
between categories of bi(co)modules in the category $\SSeqnu(\Sp)$ of symmetric sequences in spectra concentrated in positive arities. 
This is an equivalence of categories by a theorem of Heuts \cite{heuts2024koszulduality}.
% The duality can be described as follows. Given a reduced\footnote{Reduced means that $\sO(0)\simeq \ast$ and $\sO(1) \simeq \Sph$.} operad $\sO$ in spectra, we can form the $2$-sided bar construction $B(\unit, \sO, \unit)$, given by the colimit of the usual simplicial diagram
%  \[
% \begin{tikzcd}
%       \cdots \unit \circ \sP \circ \sP \circ \unit \ar[r, shift left=2] \ar[r] \ar[r, shift right=2] & \unit \circ \sP \circ \unit \ar[l, shift left, shorten=0.4em] \ar[l, shift right, shorten=0.4em] \ar[r, shift left] \ar[r, shift right] & \unit \ar[l, shorten=0.4em]
% \end{tikzcd}
% \]
% in the $\infty$-category $\SSeq(\Sp)$ of symmetric sequences in spectra.
% The symbol $\circ$ denotes the composition product of symmetric sequences, and $\unit$ is the unit of this monoidal structure, given by the sphere spectrum $\Sph$ concentrated in arity $1$.
% It turns out that this bar construction can be equipped with the structure of a cooperad in spectra, which we denote by $B\sO$ and call the Koszul dual of $\sO$.
% Starting with cooperad $\sQ$, one can form the 

It is the purpose of this note to explain that all these equivalences can be assembled into a single equivalence of $(\infty, 2)$-categories.
More precisely, Blom \cite{blom2025straighteningfunctor} constructed a Morita $(\infty, 2)$-category $\MMor_{\Op}$ with objects given by strongly positive operads in spectra and $1$-morphisms given by bimodules.
If $M$ and $N$ are respectively an $(\sO, \sP)$- and a $(\sP, \sQ)$-bimodule, then their composition in $\MMor_{\Op}$ is given by the relative composition product $M \circ_\sP N$, which is defined as the colimit of the usual simplicial diagram
 \[
\begin{tikzcd}
      \cdots M \circ \sP \circ \sP \circ N \ar[r, shift left=2] \ar[r] \ar[r, shift right=2] & M \circ \sP \circ N \ar[l, shift left, shorten=0.4em] \ar[l, shift right, shorten=0.4em] \ar[r, shift left] \ar[r, shift right] & M \circ N. \ar[l, shorten=0.4em]
\end{tikzcd}
\]
Here $\circ$ denotes the composition product of symmetric sequences.
Similarly, one can also define a dual Morita $(\infty, 2)$-category $\MMor_{\coOp}$, which has cooperads as its objects and bicomodules for $1$-morphisms.
Composition in this category is given by the cobar construction of bicomodules.
The main theorem of this paper is:

\begin{introtheorem} \label{introthm: main-theorem}
    There is an equivalence of $(\infty, 2)$-categories $\MMor_{\Op} \simeq \MMor_{\mathrm{coOp}}$ that sends an operad $\sO$ to its Koszul dual cooperad $B\sO$, and is given on mapping categories by the Koszul duality functor
    \[
    B \colon \BMod_{(\sO, \sP)}(\SSeqnu(\Sp)) \xrightarrow{\sim} \BcoMod_{(B\sO, B\sP)}(\SSeqnu(\Sp)).
    \]
\end{introtheorem}

\begin{remark}
    In fact, we will prove a slightly more general result which applies not just to (co)operads in spectra but also to (co)algebras in the $\infty$-category of functor symmetric sequences $\SSeq(\sA, \sA)$, where $\sA$ is a stable presentable $\infty$-category.
\end{remark}

\subsection*{Acknowledgments}
The author wishes to thank Thomas Blom and Gijs Heuts for many insightful conversations. He was supported by the Royal Society through grant URF\textbackslash R1\textbackslash 211075 and is grateful to the University of Oxford for its hospitality. 

\subsection*{Notations and conventions}

\begin{enumerate}[(\arabic*)]
    \item We will write category for $(\infty, 1)$-category and $2$-category for $(\infty, 2)$-category.
    \item We will write $\Cat$ for the $2$-category of categories.
    \item If $\sX$ is a $2$-category and $x, y$ are objects of $\sX$, we write $\sX(x, y)$ for the category of $1$-morphisms from $x$ to $y$.
\end{enumerate}

\section{Morita categories}

Blom showed in \cite{blom2025straighteningfunctor} how to construct the Morita $2$-category $\MMor(\sX)$ for sufficiently well-behaved $2$-categories $\sX$. It is characterized by a universal property in terms of lax functors, as we will now explain.

% whose objects are pairs $(x, \sO)$, where $x$ is an object of $\sX$ and $\sO$ is an algebra in the monoidal endomorphism category $\sX(x, x)$.
% A $1$-morphism from $(x, 
% \sO)$ to $(y, \sP)$ in $\MMor(\sX)$ is given by a $(\sP, \sO)$-bimodules in the mapping category $\sX(x, y)$.
% % If $M$ is a morphism from $(y, \sP)$ to $(z, \sQ)$, and $N$ is a morphism from $(x, \sO)$ to $(y, \sP)$,
% If we are given morphisms
% \[
% M \in \MMor((y, \sP), (z, \sQ)) \quad \text{and} \quad N \in \MMor((x, \sO), (y, \sP))
% \]
% then their composition is given by the relative bar construction $M \circ_\sP N$, which is defined as the colimit of the usual simplicial diagram
% \[
%     \begin{tikzcd}
%           \cdots M \circ \sP \circ \sP \circ N \ar[r, shift left=2] \ar[r] \ar[r, shift right=2] & M \circ \sP \circ N \ar[l, shift left, shorten=0.4em] \ar[l, shift right, shorten=0.4em] \ar[r, shift left] \ar[r, shift right] & M \circ N, \ar[l, shorten=0.4em]
%    \end{tikzcd}
%     \]

A lax functor $F \colon \sX \to \sY$ between $2$-categories should be thought of as a functor that preserves units and composition up to a possibly non-invertible $2$-morphism.
Part of the data of such a functor consists of $2$-morphisms
\[
\eta_w \colon \id_F(w) \Rightarrow F(\id_w) \quad \text{and} \quad \mu_{f, g} \colon F(f) \circ F(g) \Rightarrow F(f \circ g)
\]
for every object $w$ and every pair of composable morphisms $x \xrightarrow{g} y \xrightarrow{f} z$ in $\sX$.
We call $\mu_{f, g}$ a lax comparison map of $F$.
A lax functor is called \emph{unital} if the $2$-morphisms $\eta_w$ are invertible for all objects $w$.
A unital lax functor is a (strong) functor of $2$-categories precisely if the $2$-morphisms $\mu_{f, g}$ are invertible for all $f$ and $g$.
There is also the notion of a (unital) oplax functor, for which the morphisms $\eta_w$ and $\mu_{f, g}$ point in the other direction.
We can define categories $\Lax(\sX, \sY)$ and $\Oplax(\sX, \sY)$ of lax and oplax functors from $\sX$ to $\sY$.
See \cite[\S A.3.2]{blansblom2025chainrulegoodwilliecalculus} for precise definitions.

% Up to some technical fine print, the Morita category of a $2$-category $\sY$ is a $2$-category $\MMor(\sY)$ together with a lax functor $\MMor(\sY) \to \sY$ characterized by the universal property that for every lax functor $F \colon \sX \to \sY$, there is an essentially unique unital lax functor $F^u \colon \sX \to \MMor(\sY)$ making the diagram
% \[
% \begin{tikzcd}
%     & \MMor(\sY) \ar[d] \\
%     \sX \ar[ur, "F^u"] \ar[r, "F"'] & \sY& 
% \end{tikzcd}
% \]
% commute.

\begin{remark}
    In order to correctly state the universal property of the Morita $2$-category, we recall that a \emph{$2$-precategory} is a Segal object
    \[
    \sX \colon \Delta^\op \to \Cat
    \]
    such that $\sX_0$ is a space.
    A $2$-precategory is a $2$-category precisely if the Segal object defining it is complete.
    By \cite[Proposition 7.16]{Haugseng2015RectificationEnrichedCategories}, the inclusion of $2$-precategories into $2$-categories admits a left adjoint, called completion.
    Almost all concepts defined for $2$-categories have a counterpart for $2$-precategories.
    See \cite[\S A.1]{blansblom2025chainrulegoodwilliecalculus} for more details.
\end{remark}

\begin{theorem}[{\cite[Proposition 4.4.6]{blansblom2025chainrulegoodwilliecalculus}}] \label{Thm: Morita-category-exists}
    Let $\sX$ be a $2$-precategory and suppose that:
    \begin{enumerate}[\upshape{(}\arabic*\upshape{)}]
        \item For any pair of objects $x, y \in \sX$ the mapping category $\sX(x, y)$ admits geometric realizations; 
        \item For any three objects $x, y, z \in \sX$, the composition functor
        \[
        \sX(y,z) \times \sX(x, y) \to \sX(x, z)
        \]
        preserves geometric realizations in both variables.
    \end{enumerate}
    Then there exists a $2$-precategory $\MMor^{\mathrm{pre}}(\sX)$ together with a lax functor $\MMor^{\mathrm{pre}}(\sX) \to \sX$ such that for every $2$-precategory $\sY$, postcomposition with this functor induces an equivalence of spaces
    \[
    \Laxunit(\sY, \MMor^{\mathrm{pre}}(\sX))^{\simeq} \xrightarrow{\sim} \Lax(\sY, \sX)^{\simeq}.
    \]
\end{theorem}

\begin{warning}
    The universal property of the previous theorem does not hold for categories of (unital) lax functors, but only for their groupoid cores.
    To obtain a universal property that holds on the level of categories, one should work in the context of lax functors of double categories instead of $2$-categories.
\end{warning}

\begin{definition}
    Suppose $\sX$ is a $2$-category that satisfies the conditions of the previous theorem.
    We define the \emph{Morita $2$-category} $\MMor(\sX)$ of $\sX$ to be the completion of $\MMor^{\mathrm{pre}}(\sX)$.
\end{definition}

\begin{remark}
    The category $\MMor(\sX)$ can be described explicitly as follows.
    Its objects are pairs $(x, \sO)$, where $x$ is an object of $\sX$ and $\sO$ is an algebra in the monoidal endomorphism category $\sX(x, x)$.
    Mapping categories in $\MMor(\sX)$ can be identified with categories of bimodules:
    \[
    \MMor(\sX)((x, \sO), (y, \sP)) \simeq \BMod_{(\sP, \sO)}(\sX(x, y)).
    \]
    If we are given morphisms
    \[
    M \in \MMor((y, \sP), (z, \sQ)) \quad \text{and} \quad N \in \MMor((x, \sO), (y, \sP))
    \]
    then their composition is given by the relative composition product of bimodules
    \[
    M \circ_\sP N \coloneqq \colim \left(
        \begin{tikzcd}
              \cdots M \circ \sP \circ \sP \circ N \ar[r, shift left=2] \ar[r] \ar[r, shift right=2] & M \circ \sP \circ N \ar[l, shift left, shorten=0.4em] \ar[l, shift right, shorten=0.4em] \ar[r, shift left] \ar[r, shift right] & M \circ N \ar[l, shorten=0.4em]
       \end{tikzcd}
       \right).
    \]
\end{remark}

\begin{remark} \label{rem: unital-lift}
    Suppose that $F \colon \sY \to \sX$ is a lax functor.
    Then for each object $y \in \sY$, we obtain a lax monoidal functor $F \colon \sY(y, y) \to \sX(F(y), F(y))$.
    In particular, it sends the identity $\id_y$ to an algebra 
    \[
    F(\id_y) \in \Alg(\sX(F(y), F(y))).
    \]
    On mapping categories where the source is not equivalent to the target, $F$ induces a lax functor of bitensored categories.
    In particular, for every morphism $f \colon y \to z$, its image $F(f)$ acquires the structure of an $(F(\id_z), F(\id_y))$-bimodule in $\sX(F(y), F(z))$.

    The unital lax lift $F^{\mathrm{un}} \colon \sY \to \MMor(\sX)$ coming from the universal property of the Morita category sends $y$ to the pair $(y, F(\id_y))$ and $f \colon y \to z$ to the bimodule $F(f)$ as defined above.
    Given a pair of composable morphisms $x \xrightarrow{g} y \xrightarrow{f} z$ in $\sX$,
    % , we have $f \circ g = f \circ_{\id_y} g$, in other words their composite is equivalent to the relative composition product over $\id_y$.
    we define a morphism of bimodules by combining the assembly map for colimits with the lax structure of $F$:
    \[
    F(f) \circ_{F(\id_y)} F(g) \Rightarrow F(f \circ_{id_y} g) \simeq F(f \circ g).
    \]
    This is the lax comparison map of $F^{\mathrm{un}}$ for the pair $(f, g)$.
\end{remark}

In this article, we will only consider the Morita category of the $2$-category of stable presentable categories and symmetric sequences, which we will now define.

\begin{definition}
    Let $\sA$ and $\sB$ be stable presentable categories.
    The category of \emph{nonunital symmetric sequences} from $\sA$ to $\sB$ is defined as
    \[
    \SSeqnu(\sA, \sB) \coloneqq \FunL(\Sym^{\mathrm{nu}}(\sA), \sB),
    \]
    where $\Sym^{\mathrm{nu}}(\sA) \simeq \bigoplus_{n \geq 1} \sA^{\otimes n}_{h\Sigma_n}$ denotes the free non-unital commutative algebra on $\sA$ in $\presl$ with respect to the tensor product $\otimes$ of presentable categories.
\end{definition}

Since all our symmetric sequences are nonunital, we will usually drop this adjective.

\begin{remark}
    In concrete terms, a symmetric sequence $F \in \SSeqnu(\sA, \sB)$ consists of a sequence of functors
    \[
    F_n \colon \sA^{\times n}_{h\Sigma_n} \to \sB \qquad \text{for } n \geq 1,
    \]
    such that the underlying functor of each $F_n$ preserves all colimits in each variable separately.
    We call $F_n$ the arity $n$ component of $F$.
    One can also index these components by non-empty finite sets instead of natural numbers.
\end{remark}

\begin{definition}
    We say a symmetric sequence $F \in \SSeqnu(\sA, \sA)$ is \emph{strongly positive} if $F_1 \simeq \id_\sA$.
    We write $\SSeqsp(\sA, \sA)$ for the full subcategory spanned by the strongly positive symmetric sequences.
\end{definition}

\begin{remark}
    Given symmetric sequences $F \in \SSeqnu(\sB, \sC)$ and $G \in \SSeqnu(\sA, \sB)$, one can define their \emph{composition product} $F \circ G \in \SSeqnu(\sA, \sC)$ by the formula
    \[
    (F \circ G)_I = \bigoplus_{E \in \mathrm{Part}(I)}F_E \circ \{G_J\}_{J \in E},
    \]
    where $I$ is a non-empty finite set and $\mathrm{Part}(I)$ is the set of partitions of $I$.
    The symbol $\circ$ on the right hand side denotes composition of multivariable functors.
\end{remark}

Stable presentable categories and symmetric sequences between them can be assembled into a $2$-category.

\begin{proposition}[{\cite[\S 3.2]{blansblom2025chainrulegoodwilliecalculus}}]
    There is a $2$-category $\pressymstc$ such that:
    \begin{enumerate}[\upshape{(}\arabic*\upshape{)}]
        \item The objects of $\pressymstc$ are stable presentable categories;
        \item The category of $1$-morphisms from $\sA$ to $\sB$ in $\pressymstc$ is given by $\SSeqnu(\sA, \sB)$;
        \item Composition is given by the composition product.
    \end{enumerate}
    Moreover, the composition product commutes with sifted colimits in both variables.
\end{proposition}

If $\sA \in \pressymstc$, we write $\unit_\sA$, or just $\unit$ if no confusion can arise, for the identity symmetric sequence of $\sA$.
It is given by the functor $\id_\sA$ in arity $1$ and the zero functor in all other arities.

\begin{remark}
    The previous proposition implies that for any stable presentable category $\sA$, the composition product defines a monoidal structure on $\SSeqnu(\sA, \sA)$.
    In the case $\sA = \Sp$, we have an equivalence
    \[
    \SSeqnu(\Sp, \Sp) \simeq \Fun(\Fin^\simeq_{\geq 1}, \Sp),
    \]
    where $\Fin^\simeq_{\geq 1}$ denotes the category of non-empty finite sets and bijections.
    In other words,\ $\SSeqnu(\Sp, \Sp)$ is the ordinary category of symmetric sequences in $\Sp$ concentrated in positive arities.
    In this case, the composition product agrees with the ordinary composition product of symmetric sequences, as defined in \cite[\S 4.1.2]{brantnerThesis}; see \cite[Proposition 3.2.16]{blansblom2025chainrulegoodwilliecalculus} for a comparison of these monoidal structures.
    In particular, an algebra in $\SSeqnu(\Sp)$ is the same as a non-unital operad in spectra.
\end{remark}

\begin{definition}
    Let $\sA$ be a stable presentable category.
    We call an algebra or coalgebra in the monoidal category $\SSeqnu(\sA, \sA)$ \emph{strongly positive} if its underlying symmetric sequence is.
\end{definition}

It follows from the previous proposition that the Morita category of $\pressymstc$ is well-defined.

\begin{definition}
    Let $\MMor_+$ denote the full subcategory of $\MMor(\pressymstc)$ spanned by the pairs $(\sA, \sO)$ where $\sO$ is strongly positive.
\end{definition}

We would also like to define the \emph{dual Morita category} of $\pressymstc$, which should have pairs $(\sA, \sQ)$ with $\sQ$ a coalgebra in $\SSeqnu(\sA, \sA)$ as its objects, and bicomodules as its morphisms.
If $M$ and $N$ are respectively a $(\sQ, \mathscr{R})$- and $(\mathscr{R}, \mathscr{S})$-bicomodule, then their composition in the dual Morita category should be given by the cobar construction
    \[
    M \Box_{\mathscr{R}} N \coloneqq \lim\left(
    \begin{tikzcd}
    M \circ N \ar[r, shift left] \ar[r, shift right] & M \circ \mathscr{R} \circ  N \ar[l, shorten=0.4em] \ar[r, shift left=2] \ar[r] \ar[r, shift right=2] & M \circ \mathscr{R} \circ \mathscr{R} \circ N \cdots \ar[l, shift left, shorten=0.4em] \ar[l, shift right, shorten=0.4em]
\end{tikzcd}
    \right).
    \]

In general, composition in $\pressymstc$ does not commute with cobar constructions of bicomodules, so that we can't directly define its dual Morita category. However, the following proposition shows that composition does commute with cobar constructions if we restrict to strongly positive coalgebras.

Recall that given a 2-(pre)category $\sX$, we can define a new 2-(pre)category $\sX^{\co}$ which has the same objects as $\sX$, but the direction of the $2$-morphisms is turned around:
\[
\sX^{\co}(x, y) \simeq \sX(x,y)^{\op}.
\]
One could reasonably try to define the dual Morita category as the completion of $\MMor^{\mathrm{pre}}((\pressymstc)^{\co})^\co$, but this runs into the problem that composition in $\pressymstc$ does not preserve totalizations in both variables, so that $(\pressymstc)^\co$ does not satisfy the requirements of \cref{Thm: Morita-category-exists}.
An inspection of the construction in \cite[\S 13]{blom2025straighteningfunctor} shows that it would suffice for composition in $\pressymstc$ to only preserve cobar constructions of the form $M \Box_{\mathscr{R}} N$ in both variables, but this condition is still not satisfied.
However, such totalizations are preserved by composition if we assume that $\mathscr{R}$ is strongly positive:

\begin{proposition} \label{prop: cobar-finite-totalization}
    Let $\sA$, $\sB$, and $\sC$ be stable presentable categories, and let $\sQ$, $\sR$, and $\sS$ be strongly positive coalgebras in the categories $\SSeqnu(\sA,\sA)$, $\SSeqnu(\sB, \sB)$ and $\SSeqnu(\sC, \sC)$ respectively.
    Suppose we are given bicomodules
    \[
    M \in \BcoMod_{(\sS, \sR)}(\SSeqnu(\sB, \sC)) \quad \text{and} \quad N \in \BcoMod_{(\sR, \sQ)}(\SSeqnu(\sA, \sB)).
    \]
    Then the limit defining the relative cobar construction $M \Box_{\sR} N$ is an $(n-1)$-skeletal totalization in arity $n$ for every $n \geq 1$.
    It follows that for every stable presentable category $\sD$, both composition functors
    \[
    \SSeqnu(\sA, \sC) \times \SSeqnu(\sD, \sA) \to \SSeqnu(\sD, \sC) \quad \text{and} \quad \SSeqnu(\sC, \sD) \times \SSeqnu(\sA, \sC)  \to \SSeqnu(\sA, \sD)
    \]
    preserve the limit defining the relative cobar construction $M \Box_{\sR} N$.
\end{proposition}
\begin{proof}
    This follows from \cite[Propositions 3.2.12 and 4.2.5]{blansblom2025chainrulegoodwilliecalculus}.
\end{proof}

This has the following consequence:

\begin{proposition} \label{prop: comor-plus-exists}
    There exists a $2$-precategory $\coMMor_+^{\mathrm{pre}}$ together with an oplax functor $U \colon \coMMor_+^{\mathrm{pre}} \to \pressymstc$ such that for every $2$-precategory $\sY$, composition with $U$ induces an equivalence of spaces
    \[
    \Oplaxunit(\sY, \coMMor_+^{\mathrm{pre}})^{\simeq} \xrightarrow{\sim} \Oplax_+(\sY, \pressymstc)^{\simeq}.
    \]
    Here $\Oplax_+(\sY, \pressymstc) \subseteq \Oplax(\sY, \pressymstc)$ is the full subcategory on those oplax functors $F \colon \sY \to \pressymstc$ for which the coalgebra $F(\id_y)$ is strongly positive for all $y \in \sY$.
\end{proposition}
\begin{proof}
    We will need to explain some details of the construction of the Morita double category from \cite{blom2025straighteningfunctor}.
    Given a double category $\mathbb{D}$, which is by definition a Segal object
    \[
    \mathbb{D} \colon \Delta^\op \to \Cat,
    \]
     Blom constructs a new simplicial object $\overline{\bMMor}(\mathbb{D}) \colon \Delta^\op \to \Cat$ that satisfies the universal property
    \[
    \Lax^{\mathrm{un}}(\mathbb{E}, \overline{\bMMor}(\mathbb{D})) \simeq \Lax(\mathbb{E}, \mathbb{D})
    \]
    for double categories $\mathbb{E}$ \cite[\S 11]{blom2025straighteningfunctor}.
    For this it is important that (unital) lax functors can be defined for all simplicial objects \cite[Remark 2.14]{blom2025straighteningfunctor}, since $\overline{\bMMor}(\mathbb{D})$ is in general not a Segal object.
    Blom then defines $\bMMor(\mathbb{D})$ as a certain subobject of $\overline{\bMMor}(\mathbb{D})$.
    In case $\mathbb{D}$ satisfies (analogues of) the conditions of \cref{Thm: Morita-category-exists}, he shows that this defines a Segal object satisfying the required universal property.
    If $\sX$ is a $2$-precategory satisfying these conditions, we can regard it as a double category and define $\MMor^{\mathrm{pre}}(\sX)$ as $\mathrm{Hor}(\bMMor(\sX))$, where $\mathrm{Hor}$ denotes the right adjoint to the inclusion of $2$-precategories into double categories.
    In general, objects of $\bMMor(\mathbb{D})$ are pairs of an object $x$ in $\mathbb{D}$ together with an algebra in the endomorphism category $\mathbb{D}(x, x)$.

    We define $\bMMor_+((\pressymstc)^{\mathrm{co}})$ to be the maximal subobject of the simplicial object $\bMMor((\pressymstc)^{\mathrm{co}})$ such that in degree $0$ it is the full subcategory spanned by the objects $(\sA, \sQ)$, where $\sQ$ is a strongly positive coalgebra in $\SSeqnu(\sA, \sA)$.
    By inspecting the proof of \cite[Proposition 13.10]{blom2025straighteningfunctor} and using \cref{prop: cobar-finite-totalization}, we find that its conclusion equally well applies to $\bMMor_+((\pressymstc)^{\mathrm{co}})$, so that this simplicial object satisfies the Segal condition and therefore defines a double category.
    Moreover, the proof given in \cite[\S 14]{blom2025straighteningfunctor} shows that for every double category $\mathbb{E}$, we have an equivalence of categories
    \[
    \Oplaxunit(\mathbb{E}, \bMMor_+((\pressymstc)^{\co})^\co) \simeq \Oplax_+(\mathbb{E}, \pressymstc).
    \]
    Since we have an equivalence of spaces
    \[
    \Oplax(\sX, \mathrm{Hor}(\mathbb{D}))^\simeq \simeq \Oplax(\sX, \mathbb{D})^\simeq
    \]
    for every $2$-precategory $\sX$ and double category $\mathbb{D}$, we find that 
    \[
    \coMMor^{\mathrm{pre}}_+ \coloneqq \mathrm{Hor}(\bMMor_+((\pressymstc)^{\co})^\co)
    \]
    satisfies the required universal property.
\end{proof}

\begin{definition}
    We define $\coMMor_+$ to be the completion of the $2$-precategory $\coMMor_+^{\mathrm{pre}}$.
\end{definition}

\begin{remark}
    The objects of the $2$-category $\coMMor_+$ are pairs $(\sA, \sQ)$ where $\sA$ is a stable presentable category and $\sQ$ is a strongly positive coalgebra in $\SSeqnu(\sA, \sA)$.
    Morphisms are given by bicomodules, which are not required to have underlying strongly positive symmetric sequence.
    Composition is given by the cobar construction.
\end{remark}

\begin{remark}
    If $\sY$ is a $2$-category together with an oplax functor $F \colon \sY \to \pressymstc$ for which $F(\id_y)$ is strongly positive for all $y \in \sY$, then we obtain an oplax unital functor $F^{\mathrm{un}} \colon \sY \to \coMMor_+$ as the composite
    \[
    \sY \to \coMMor_+^{\mathrm{pre}} \to \coMMor_+,
    \]
    where the first functor comes from the universal property of \cref{prop: comor-plus-exists} and the second functor is the canonical map of $\coMMor_+^{\mathrm{pre}}$ to its completion.
    The effect of $F^{\mathrm{un}}$ on objects and $1$-morphisms as well as the description of its oplax comparison maps are dual to those given in \cref{rem: unital-lift}.
\end{remark}

\section{Koszul duality}

We now recall the form of Koszul duality that we will make use of.
The following theorem deals with Koszul duality between algebras and coalgebras in symmetric sequences, and was first proved by Ching \cite[Theorem 2.15]{ChingBarCobar} for operads and cooperads in spectra.

\begin{theorem}[{\cite[Theorem 4.2.4]{blansblom2025chainrulegoodwilliecalculus}}] \label{thm:koszul-duality-operads}
    Let $\sA$ be a stable presentable category. 
    There is an adjoint equivalence of categories
    \[
    \begin{tikzcd}
        \Alg(\SSeq_+(\sA, \sA)) \ar[r, shift left, "B"] & \coAlg(\SSeq_+(\sA, \sA)). \ar[l, shift left, "C"]
    \end{tikzcd}
    \]
    On underlying symmetric sequences, we have $B\sO \simeq \unit \circ_\sO \unit$ and $C\sQ \simeq \unit \Box_\sQ \unit$.
\end{theorem}

We call $B\sO$ the Koszul dual of $\sO$ and $C\sQ$ the Koszul dual of $\sQ$.
We also have a form of Koszul duality for bi(co)modules, which was proved by Heuts:

\begin{theorem}[{\cite[Theorem 14.17]{heuts2024koszulduality}}] \label{thm:koszul-duality-bimodules}
    Let $(\sA, \sO)$ and $(\sB, \sP)$ be objects of $\MMor_+$.
    There is an adjoint equivalence
    \[
    \begin{tikzcd}
    \BMod_{(\sP, \sO)}(\SSeqnu(\sA, \sB)) \ar[r, shift left, "B"] & \BcoMod_{(B\sP, B\sO)}(\SSeqnu(\sA, \sB)) \ar[l, shift left, "C"].
    \end{tikzcd}
    \]
    On underlying symmetric sequences, we have $BM \simeq \unit \circ_\sP M \circ_\sO \unit$ and $CN \simeq \unit \Box_{B\sP} N \Box_{B\sO} \unit$.
\end{theorem}

\section{The proof}

We now come to the main result of this paper:

\begin{theorem} \label{thm: morita-comorita-equiv}
    There is an equivalence of $2$-categories
    \[
    \mathrm{KD} \colon \MMor_+ \longrightarrow \coMMor_+
    \]
    that sends an object $(\sA, \sO)$ to $(\sA, B\sO)$, and is given by the Koszul duality functor
    \[
    B \colon \BMod_{(\sP, \sO)}(\SSeqnu(\sA, \sB)) \longrightarrow \BcoMod_{(B\sP, B\sO)}(\SSeqnu(\sA, \sB))
    \]
    on mapping categories.
\end{theorem}

We begin the proof of \cref{thm: morita-comorita-equiv} by constructing the functor $\KD \colon \MMor_+ \to \coMMor_+$.
This requires us to recall some preliminaries.

First of all, we need the notion of an $\sO$-algebra for $\sO \in \Alg(\SSeqnu(\sA, \sA))$.
By \cite[Proposition 3.2.13]{blansblom2025chainrulegoodwilliecalculus}, any stable presentable category $\sA$ has a left action by $\SSeqnu(\sA, \sA)$, given by the formula
\[
F \circ x = \bigoplus_{n \geq 1} F_n(x, \ldots, x)_{h\Sigma_n}.
\]
This allows us to define the category $\Alg_\sO(\sA)$ for $\sO \in \Alg(\SSeqnu(\sA, \sA))$ as the category of left $\sO$-modules in $\sA$.
The notation stems from the case $\sA = \Sp$, since $\Alg_\sO(\Sp)$ is the usual category of algebras over the operad $\sO$.

By \cite[Proposition 6.3]{heuts2024koszulduality}, the stabilization of $\Alg_\sO(\sA)$ is equivalent to $\sA$.
We will write
\[
\begin{tikzcd}
    \Sigma^\infty_\sO \colon \Alg_\sO(\sA) \ar[r, shift left] & \sA \ar[l, shift left] \colon \Omega^\infty_\sO
\end{tikzcd}
\]
for its stabilization adjunction.
We also note that by \cite[Example 3.3.4]{blansblom2025chainrulegoodwilliecalculus}, the functor $\Omega^\infty_\sO$ preserves filtered colimits.

\begin{definition}
    A functor $F \colon \sC \to \sD$ is called \emph{reduced} if it preserves the zero object and \emph{finitary} if it preserves filtered colimits.
    We let $\Fun^{\ast, \omega}(\sC, \sD)$ denote the category of reduced finitary functors $\sC \to \sD$.
    We will write $\diffalg$ for the locally full subcategory\footnote{If $\sX$ is a subcategory of a $2$-category $\sY$, we say it is locally full if the inclusion $\sX \hookrightarrow \sY$ induces a fully faithful functor on mapping categories.} of the $2$-category $\Cat$ spanned by categories of the form $\Alg_\sO(\sA)$ for $\sO \in \Alg(\SSeq_+(\sA, \sA))$ and reduced finitary functors between them.
\end{definition}

In \cite[Definition 2.26]{blansblom2026productrulegoodwilliecalculus}, we define a functor of $2$-categories 
\[
\reAlg \colon \MMor_+ \to \diffalg
\]
that sends $(\sA, \sO)$ to the category $\Alg_\sO(\sA)$, and a bimodule $M \in \BMod_{(\sP, \sO)}(\SSeqnu(\sA, \sB))$ to the functor $\Alg_\sO(\sA) \to \Alg_\sP(\sB)$ given by the formula
\[
X \mapsto M \circ_\sO X = \colim(
    \begin{tikzcd}
          \cdots M \circ \sO \circ \sO \circ X \ar[r, shift left=2] \ar[r] \ar[r, shift right=2] & M \circ \sO \circ X \ar[l, shift left, shorten=0.4em] \ar[l, shift right, shorten=0.4em] \ar[r, shift left] \ar[r, shift right] & M \circ X. \ar[l, shorten=0.4em]
   \end{tikzcd} 
   ).
\]

We will also make use of some concepts and results from Goodwillie calculus for the construction of $\KD$.
Given a reduced finitary functor $F \colon \Alg_\sO(\sA) \to \Alg_\sP(\sB)$, its \emph{Goodwillie derivatives} constitute a symmetric sequence
\[
\partial_*F \in \SSeqnu(\sA, \sB).
\]
It follows from \cite[Theorem 3.3.2]{blansblom2025chainrulegoodwilliecalculus} and \cite[Theorem 5.7]{blansblom2026productrulegoodwilliecalculus} that $\partial_*F$ admits the structure of a $(\sP, \sO)$-bimodule.
In fact, the assignment $F \mapsto \partial_*F$ can be extended to a functor into the Morita category.

\begin{theorem}[{\cite[Theorem 4.4.7]{blansblom2025chainrulegoodwilliecalculus}}]
    There is a functor of $2$-categories
    \[
    \partial_* \colon \diffalg \to \MMor_+
    \]
    that sends the object $\Alg_\sO(\sA)$ to $(\sA, \sO)$ and a reduced finitary functor $F \colon \Alg_\sO(\sA) \to \Alg_\sP(\sB)$ to the $(\sP, \sO)$-bimodule $\partial_*F$.
\end{theorem}

\begin{remark}
    We write $\diffst$ for the full subcategory of $\diffalg$ spanned by the stable presentable categories; i.e.\ those of the form $\Alg_{\unit_\sA}(\sA) \simeq \sA$.
    Note that the category $\pressymstc$ is equivalent to the full subcategory of $\MMor_+$ on objects of the form $(\sA, \unit_\sA)$ and that the restriction of the functor $\partial_* \colon \diffalg \to \MMor_+$ to $\diffst$ factors through $\pressymstc$.
\end{remark}

We will need the following result on the interaction between $\partial_*$ and $\reAlg$.

\begin{proposition} \label{prop: der-section-lambda-alg}
    Let $(\sA, \sO)$ and $(\sB, \sP)$ be objects of $\MMor_+$.
    The composite
    \[
    \BMod_{(\sP, \sO)}(\SSeqnu(\sA, \sB)) \xrightarrow{\reAlg} \Fun^{\ast, \omega}(\Alg_\sO(\sA), \Alg_\sP(\sB)) \xrightarrow{\partial_*} \BMod_{(\sP, \sO)}(\SSeqnu(\sA, \sB))
    \]
    induced by $\reAlg$ and $\partial_*$ on mapping categories is equivalent to the identity functor.
\end{proposition}
\begin{proof}
    This follows from \cite[Proposition 3.18]{blansheuts2026characterizationspectrallieoperad}.
\end{proof}

The final ingredient needed for the proof of \cref{thm: morita-comorita-equiv} is the following lemma.

\begin{lemma}
    There is an oplax functor of $2$-categories $\theta \colon \diffalg \to \diffst$ that sends $\Alg_\sO(\sA)$ to $\sA$ and a reduced finitary functor $F \colon \Alg_\sO(\sA) \to \Alg_\sP(\sB)$ to the composite
    \[
    \sA \xrightarrow{\Omega^\infty_\sO} \Alg_\sO(\sA) \xrightarrow{F} \Alg_\sP(\sB) \xrightarrow{\Sigma^\infty_\sP} \sB.
    \]
    If $G \colon \Alg_\sP(\sB) \to \Alg_\sQ(\sC)$ is another reduced finitary functor, then the comparison map
    \[
    \Sigma^\infty_\sQ G F \Omega^\infty_\sO \Rightarrow \Sigma^\infty_\sQ G \Omega^\infty_\sP \Sigma^\infty_\sP F \Omega^\infty_\sO
    \]
    of $\theta$ is given by the unit of the adjunction $(\Sigma^\infty_\sP, \Omega^\infty_\sP)$.
\end{lemma}
\begin{proof}
    Let $\diffalg^0$ and $\diffalg^1$ be the spaces of objects and morphisms in $\diffalg$.
    We begin by constructing a map $\alpha \colon \diffalg^0 \to \diffalg^1$ that sends $\Alg_\sO(\sA)$ to $\Sigma^\infty_\sO$.
    
    Write
    \[
    \iota \colon \Exc_1^{\ast, \omega}(\Alg_\sO(\sA), \Alg_\sP(\sB)) \hookrightarrow \Fun^{\ast, \omega}(\Alg_\sO(\sA), \Alg_\sP(\sB))
    \]
    for the inclusion of the full subcategory spanned by the linear functors in the sense of Goodwillie calculus; recall that a functor is linear if it is reduced and sends pushout squares to pullback squares.
    The functor $\iota$ has a left adjoint $\lin$ given by the $1$-excisive approximation of Goodwillie \cite[Theorem 6.1.1.10]{HA}.
    By the equation
    \[
    \lin(F \circ G) \simeq \lin(\lin(F) \circ \lin(G))
    \]
    of Arone--Ching \cite[Proposition 1.3.1]{AroneChingChainRule}, it follows that the left adjoints $\lin$ for varying $\Alg_\sO(\sA)$ and $\Alg_\sP(\sB)$ form a compatible family of localizations \cite[Definition A.3.24]{blansblom2025chainrulegoodwilliecalculus}, so that by \cite[Proposition A.3.25]{blansblom2025chainrulegoodwilliecalculus} we obtain a $2$-category $\diffalg^{\mathrm{lin}}$ which has the same space of objects as $\diffalg$ and linear functors as $1$-morphisms. 
    Composition in $\diffalg^{\mathrm{lin}}$ is defined by taking functor composition and then applying $\lin$.
    The same proposition also provides us with functors of $2$-categories
    \[
    \lin \colon \diffalg \rightleftarrows \diffalg^{\mathrm{lin}} \colon \iota
    \]
    where $\lin$ is strong and $\iota$ is lax.
    They induce the identity on objects and the adjunction $(\lin, \iota)$ on mapping categories.

    The inclusion $\diffst^{\mathrm{lin}} \hookrightarrow \diffalg^{\mathrm{lin}}$ of the full subcategory spanned by the stable categories is a right adjoint, since by \cite[Theorem 6.2.3.21]{HA} precomposition with $\Sigma^\infty_\sO$ gives an equivalence
    \[
    \Exc_1^{\ast, \omega}(\Alg_\sO(\sA), \sB) \xrightarrow{\sim} \Exc_1^{\ast, \omega}(\sA, \sB).
    \]
    The unit of this adjunction together with the map induced on spaces of morphisms by the lax functor $\iota$ gives a map of spaces $\alpha \colon \diffalg^0 \to \diffalg^1$ that sends $\Alg_\sO(\sA)$ to the morphism $\Sigma^\infty_\sO \colon \Alg_\sO(\sA) \to \sA$.

    Since all morphisms in the image of $\alpha$ are left adjoints, we can invoke \cite[Theorem 2.1.16]{blansblom2025chainrulegoodwilliecalculus}
    to conclude that there exists an oplax functor of $2$-categories
    \[
    \theta \colon \diffalg \to \diffst
    \]
    that sends $\Alg_{\sO}(\sA)$ to $\sA$, and a functor $F \colon \Alg_{\sO}(\sA) \to \Alg_{\sP}(\sB)$ to the composite $\Sigma^\infty_\sP \circ F \circ \Omega^\infty_\sO$, such that the oplax comparison maps are given by the units of the stabilization adjunctions.
\end{proof}

We can now construct the $2$-categorical Koszul duality functor.

\begin{construction} \label{constr: KD-construction}
Consider the oplax functor defined by the composite
\[
\begin{tikzcd}
\MMor_+ \ar[r, "\reAlg"] & \diffalg \ar[r, "\theta"] & \diffst \ar[r, "\partial_*"] & \pressymstc.
\end{tikzcd}
\]
Unraveling the definition, we see that it sends the identity morphism of $(\sA, \sO) \in \MMor_+$ to $\partial_*{\Sigma^\infty_\sO \Omega^\infty_\sO}$, which is strongly positive by \cite[Example 3.1.36]{blansblom2025chainrulegoodwilliecalculus}.
By the universal property of the dual Morita category \cref{prop: comor-plus-exists}, this composite therefore lifts to a unital oplax functor 
\[
\KD \colon \MMor_+ \to \coMMor_+.
\]
\end{construction}

We now show that this functor behaves in the expected way on objects and $1$-morphisms.

\begin{proposition}\label{prop: KD-on-objects-and-morphisms}
    The functor $\KD \colon \MMor_+ \to \coMMor_+$ sends an object $(\sA, \sO)$ to $(\sA, B\sO)$, where $B\sO$ denotes the Koszul dual of $\sO$.
    On mapping categories, it is given by the Koszul duality functor
    \[
    B \colon \BMod_{(\sP, \sO)}(\SSeqnu(\sA, \sB)) \longrightarrow \BcoMod_{(B\sP, B\sO)}(\SSeqnu(\sA, \sB))
    \]
\end{proposition}
\begin{proof}
Tracing through the construction, we find that 
\[
\KD(\sA, \sO) = (\sA, \partial_*{\Sigma^\infty_{\sO}\Omega^\infty_{\sO}}).
\]
Since there is an equivalence of coalgebras $\partial_*{\Sigma^\infty_{\sO}\Omega^\infty_{\sO}} \simeq B\sO$ by 
\cite[Theorem 4.3.1]{blansblom2025chainrulegoodwilliecalculus} and \cite[Theorem 5.7]{blansblom2026productrulegoodwilliecalculus}, we have proved the first claim.

Suppose $(\sA, \sO)$ and $(\sB, \sP)$ are a pair of objects in $\MMor_+$.
On mapping categories, the functor $\KD$ induces the composite

\begin{align*}
    \BMod_{(\sP, \sO)}(\SSeqnu(\sA, \sB)) & \myarrow{\reAlg} \Fun^{\ast, \omega}(\Alg_\sO(\sA), \Alg_{\sP}(\sB)) \\
    & \myarrow{\Sigma^\infty_{\sP} \circ - \circ \Omega^\infty_{\sO}} \BcoMod_{(\Sigma^\infty_{\sP} \Omega^\infty_{\sP}, \Sigma^\infty_{\sO} \Omega^\infty_{\sO})}(\Fun^{\ast, \omega}(\sA, \sB))  \\
    & \myarrow{\partial_*} \BcoMod_{(B\sP, B\sO)}(\SSeqnu(\sA, \sB)),
\end{align*}
where we used the equivalences $\partial_*{\Sigma^\infty_{\sO} \Omega^\infty_{\sO}} \simeq B\sO$ and $\partial_*{\Sigma^\infty_{\sP} \Omega^\infty_{\sP}} \simeq B\sP$ to identify the target of the final morphism.
By \cite[Corollary 4.4.4]{blansblom2025chainrulegoodwilliecalculus}, the composite of the final two arrows is equivalent to
\[
\Fun^{\ast, \omega}(\Alg_\sO(\sA), \Alg_{\sP}(\sB)) \xrightarrow{\partial_*} \BMod_{(\sP, \sO)}(\SSeqnu(\sA, \sB)) \xrightarrow{B} \BcoMod_{(B\sP, B\sO)}(\SSeqnu(\sA, \sB)).
\]
Since $\partial_* \circ \reAlg$ is equivalent to the identity on mapping categories by \cref{prop: der-section-lambda-alg}, this completes the proof.
\end{proof}

The following lemma is the key computation needed for the proof of the main theorem.

\begin{lemma}\label{lem: KD-strong}
    The oplax functor $\KD \colon \MMor_+ \to \coMMor_+$ is a strong functor of $2$-categories.
\end{lemma}
\begin{proof}
    Given bimodules $M \in \MMor((\sB, \sP), (\sC, \sQ))$ and $N \in \MMor((\sA, \sO), (\sB, \sP))$, we need to show that the natural map
    \[
    B(M \circ_\sP N) \to B M \Box_{B\sP} BN
    \]
    of $(B\sO, B\sQ)$-bicomodules induced by $\KD$ is an equivalence.
    Since both source and target of this map commute with relative bar constructions in both variables by \cref{prop: cobar-finite-totalization}, it suffices to show the map is an equivalence in case $M$ and $N$ are free bimodules.
    
    So suppose $M =\sQ \circ X \circ \sP$ and $N = \sP \circ Y \circ \sO$.
    The source of the map becomes
    \[
    B(\sQ \circ X \circ \sP \circ_\sP \sP \circ Y \circ \sO) \simeq B(\sQ \circ X \circ \sP \circ Y \circ \sO) \simeq \triv(X \circ \sP \circ Y),
    \]
    where we used that Koszul duality sends free bimodules to trivial bicomodules e.g.\ by \cite[\SS 6-7]{heuts2024koszulduality}.
    On the other hand, the target of the map becomes
    \begin{align*}
        B(\sQ \circ X \circ \sP) \Box_{B\sP} B(\sP \circ Y \circ \sO) & \simeq \triv(X) \Box_{B\sP} \triv(Y) \\
        & \simeq \triv(X) \circ (\unit \Box_{B\sP} \unit) \circ \triv(Y) \\
        & \simeq \triv(X \circ \sP \circ Y).
    \end{align*}
    Here, we used that the functor $B$ takes free bimodules to trivial bicomodules, and that $\unit \Box_{B\sP} \unit $ is equivalent to the Koszul dual operad of $B\sP$, which is $\sP$.
    
    So both source and target of the map are given by $\triv(X \circ \sP \circ Y)$, and a simple diagram shows that it is an equivalence.
\end{proof}

\begin{remark}
    A special case of the equivalence $B(M \circ_\sP N) \simeq B M \Box_{B\sP} BN$ was proved by Arone--Ching in \cite[Theorem 4.5.2]{AroneChingChainRule}.
\end{remark}

\begin{proof}[Proof of \cref{thm: morita-comorita-equiv}]
    The functor $\KD \colon \MMor_+ \to \coMMor_+$ from \cref{constr: KD-construction} has the right effect on objects and mapping categories by \cref{prop: KD-on-objects-and-morphisms}. 
    It follows from Koszul duality for operads (\cref{thm:koszul-duality-operads}) that it is essentially surjective and from Koszul duality for bimodules (\cref{thm:koszul-duality-bimodules}) that it is fully faithful.
    Since it is a strong functor by \cref{lem: KD-strong}, this implies that it is an equivalence of $2$-categories.
\end{proof}

\begin{remark}
    By restricting the equivalence from this theorem to the full subcategory of $\MMor_+$ spanned by objects of the form $(\Sp, \sO)$, we obtain the $2$-categorical form of Koszul duality for spectral operads and bimodules stated in the introduction.
\end{remark}

\phantomsection
\printbibliography 

\end{document}